\documentclass[12pt]{amsart}
 
\usepackage{amssymb, amscd, txfonts}
\usepackage{graphicx}
\usepackage[all]{xy} 

\numberwithin{equation}{section}

\newtheorem{theorem}{Theorem}[section]
\newtheorem{proposition}[theorem]{Proposition}
\newtheorem{lemma}[theorem]{Lemma}
\newtheorem{corollary}[theorem]{Corollary}

\theoremstyle{definition}

\theoremstyle{remark}
\newtheorem{remark}[theorem]{Remark}

\renewcommand{\ker}{\operatorname{Ker}}

\newcommand{\Z}{\mathbb{Z}}
\newcommand{\Q}{\mathbb{Q}}

\newcommand{\C}{\mathbb{C}}
\newcommand{\proj}{{\mathbb P}}
\newcommand{\QZ}{\mathbb{Q}/\mathbb{Z}}

\newcommand{\Griff}{{\rm Griff}_1(X)}

\newcommand{\NS}{{\rm NS}_1(X)}
\newcommand{\T}{\mathcal{T}_1(X)}
\newcommand{\CH}{{\rm CH}}

\newcommand{\HH}{\mathcal{H}}
\newcommand{\HdZ}{\mathcal{H}^{d}(\mathbb{Z})}
\newcommand{\Hdn}{\mathcal{H}^{d}({\Z}/n)}
\newcommand{\HdQZ}{\mathcal{H}^{d}(\mathbb{Q}/\mathbb{Z})}

\DeclareMathOperator{\coker}{Coker}

\begin{document}

%%%%%%% Title %%%%%%%%%%%%%%%%%%%%%%%%
\title[]{Torsion 1-cycles and the coniveau spectral sequence}
\author[]{Shouhei Ma}
\thanks{Supported by Grant-in-Aid for Scientific Research (S) 15H05738.} 
\address{Department~of~Mathematics, Tokyo~Institute~of~Technology, Tokyo 152-8551, Japan}
\email{ma@math.titech.ac.jp}
%\dedicatory{}

\begin{abstract}
We relate the torsion part of the Abel-Jacobi kernel in the Griffiths group of $1$-cycles 
to a birational invariant analogous to the degree $4$ unramified cohomology 
and an invariant associated to the generalized Hodge conjecture in degree $2{\dim}(X)-3$. 
We also describe in terms of ${\HH}$-cohomology 
the Griffiths group of $1$-cycles 
and the group of torsion cycles algebraically equivalent to zero of arbitrary dimension. 
\end{abstract} 

\maketitle

%%%%%%%%%
%%%Introduction
%%%%%%%%%

\section{Introduction}\label{sec:intro}

In this paper we are interested in a connection between algebraic cycles and birational invariant, 
or more specifically, 
a torsion subgroup of the Griffiths group of $1$-cycles and a certain cohomology group which can be thought of as 
a "homology counterpart" of the degree $4$ unramified cohomology. 

Let $X$ be a smooth complex projective variety of dimension $d$. 
Let ${\Griff}={\CH}_1(X)_{hom}/A_1(X)$ be the Griffiths group of $1$-cycles, 
the group of $1$-cycles homologous to zero modulo algebraic equivalence. 
Let ${\T}\subset{\Griff}$ be the image of the group of torsion $1$-cycles which is homologous to zero and whose Abel-Jacobi invariant is also zero. 
This subtle invariant, measuring deviation of torsion $1$-cycles with null Deligne cycle class from $A_1(X)$, 
was first introduced by Voisin \cite{Vo2}. 
%When ${\CH}_0(X)$ is supported in dimension $0$, e.g., $X$ rationally connected, 
%algebraic equivalence for such $1$-cycles coincides with rational equivalence 
%(Proposition \ref{prop:alg=rat equiv}). 

On the other hand, for an abelian group $A$, let ${\HH}^q(A)$ be the Zariski sheaf on $X$ associated to the presheaf $U\mapsto H^q(U, A)$.  
The cohomology group $H^{d-k}(X, {\HH}^d(A))$ is a birational invariant of smooth projective varieties (\cite{CT-V}) 
which is analogous to the $k$-th unramified cohomology. 
It is known that the natural homomorphism 
\begin{equation*}
H^{d-k}(X, {\HdZ})\otimes{\QZ} \to H^{d-k}(X, {\HdQZ}) 
\end{equation*}
is injective (\cite{CT-V}). 
We will establish a link between ${\T}$ and this quotient group for $k=4$. 

Recall also that the coniveau filtration of $H^k(X, {\Z})$ is defined by (\cite{Gro}, \cite{Gro2}, \cite{B-O}) 
\begin{equation*}
N^rH^k(X, {\Z}) =  {\rm Ker} ( \: H^k(X, {\Z}) \to \varinjlim_{W\subset X}H^k(X-W, {\Z}) \: ), 
\end{equation*}
where $W$ ranges over closed algebraic subsets of $X$ of codimension $\geq r$. 
We will be concerned with the case $k=2d-3$,  
where the coniveau filtration has at most two steps by the Bloch-Ogus theory \cite{B-O}: 
\begin{equation*}
0= N^{d-1}H^{2d-3}(X, {\Z}) \; \subset \; N^{d-2}H^{2d-3}(X, {\Z}) \; \subset \; N^{d-3}H^{2d-3}(X, {\Z}) = H^{2d-3}(X, {\Z}). 
\end{equation*}
We write 
\begin{equation*}
\Lambda(X) = H^{2d-3}(X, {\Z})/N^{d-2}H^{2d-3}(X, {\Z}) 
\end{equation*}
and let ${}_{tor}\Lambda(X)$ be its torsion part. 
The group $\Lambda(X)$ gives a proper analogue of 
Grothendieck's sup\'erieur cohomological invariant (\cite{Gro} \S 9) in degree $2d-3$. 
%(so to say, sup\'erieur homological invariant of degree $3$). 
His generalized Hodge conjecture (\cite{Gro2}) predicts that 
$N^{d-2}H^{2d-3}(X, {\Q})$ would be the largest sub ${\Q}$-Hodge structure contained in $H^{d-1,d-2}(X)\oplus H^{d-2,d-1}(X)$. 
If this holds, ${}_{tor}\Lambda(X)$ can be thought of as measuring defect of its naive integral version.

\begin{theorem}[Theorem \ref{thm:main}]\label{main}
Let $X$ be a smooth complex projective variety of dimension $d$. 
We have a short exact sequence 
\begin{equation*}
0  \to {}_{tor}\Lambda(X) 
    \to  H^{d-4}(X, {\HdQZ})/H^{d-4}(X, {\HdZ})\otimes{\QZ} 
    \to {\T} \to 0. 
\end{equation*}
When ${\CH}_0(X)$ is supported in dimension $\leq2$, e.g., when $X$ is rationally connected,  
then $\Lambda(X)$ is finite, ${\T}$ coincides with ${\Griff}$, and we have an exact sequence 
\begin{equation*}
0  \to \Lambda(X)  \to  H^{d-4}(X, {\HdQZ}) \to {\Griff} \to 0. 
\end{equation*}
\end{theorem}

%In the course we relate the Griffiths group ${\Griff}$ to $H^{d-3}(X, {\HdZ})$ (Proposition \ref{prop:Griff1 and H-cohomology}).   
%the following description of the Griffiths group (Proposition \ref{prop:Griff1 and H-cohomology}): 
%\begin{equation*}
%0 \to \Lambda(X) \to H^{d-3}(X, {\HdZ}) \to {\Griff} \to 0. 
%\end{equation*}

In \S \ref{sec:nAk(X)}, independently of Theorem \ref{main}, 
we study the $n$-torsion part ${}_{n}A^p(X)$ of the group $A^p(X)\subset{\CH}^p(X)$ 
of codimension $p$ cycles algebraically equivalent to zero with $p$ arbitrary. 
We deduce an exact sequence (Theorem \ref{prop: torsion A^p(X) and H-cohomology}) 
\begin{equation*}
0 \to H^{p-1}(X, \mathcal{K}_p)/n \to H^{p-1}(X, {\HH}^p({\Z}))/n \to {}_{n}A^p(X) \to 0, 
\end{equation*}
where $\mathcal{K}_p$ is the Zariski sheaf on $X$ associated to the Quillen $K$-theory 
(thanks to the work of Kerz \cite{Ke}, we may also replace it by the Milnor $K$-theory sheaf). 
For $p=d-1$ this is complementary to Theorem \ref{main}, 
controlling by $\mathcal{H}$-cohomology the group of torsion $1$-cycles in ${\CH}_1(X)_{hom}$ 
whose Abel-Jacobi invariant is contained in the algebraic part of the intermediate Jacobian.

The interaction between algebraic cycles and birational invariant has been originated from the work of Bloch-Ogus \cite{B-O}, 
and developed around the nineties by %many authors such as 
Parimala, Colliot-Th\'el\`ene, Barbieri-Viale, Kahn and others 
(see, e.g., \cite{Pa}, \cite{CT1}, \cite{BV1}, \cite{BV2}, \cite{Ka} and references in \cite{CT1}). 
Especially, the relationship between degree $3$ unramified cohomology $H^3_{nr}(X)=H^0(X, {\HH}^3)$ and 
codimension $2$ cycles was extensively studied. 
After the Bloch-Kato conjecture was proved by Rost-Voevodsky \cite{Voe}, 
Colliot-Th\'el\`ene and Voisin \cite{CT-V} revisited this subject and established a link between $H^3_{nr}(X, {\QZ})$ and 
the defect of the integral Hodge conjecture for codimension $2$ cycles. 
Voisin \cite{Vo2} extended this result to the degree $4$ unramified cohomology $H^4_{nr}(X, {\QZ})$, 
relating it to $\mathcal{T}^3(X)$. %under the condition that $H^5(X, {\Z})/N^2H^5(X, {\Z})$ being torsion-free. 
Colliot-Th\'el\`ene and Voisin \cite{CT-V} also found a relation between  
the defect of the integral Hodge conjecture for $1$-cycles and $H^{d-3}(X, {\HdQZ})$. 
Theorem \ref{main} is the ``homology'' counterpart of Voisin's interpretation of $H^4_{nr}$, 
and is also the extension of the second theorem of Colliot-Th\'el\`ene--Voisin to degree $4$. 
Note that in dimension $d=4$, Theorem \ref{main} recovers Voisin's result. 
We can also give a refinement of Voisin's result (Remark \ref{remark}). 

We would like to thank the referee for many valuable suggestions.  
In particular, Proposition \ref{prop:alg=rat equiv} was taught by the referee.

\vspace{3mm}

\textbf{Notation.} 
If $A$ is an abelian group and $n>0$ a natural number, 
we denote $A/n=A/nA$ and ${}_{n}A = \{ x\in A | nx=0 \}$. 
We write ${}_{tor}A$ for the subgroup of torsion elements. 
If $f:A\to B$ is a homomorphism, ${}_{n}f:{}_{n}A\to {}_{n}B$ and ${}_{tor}f:{}_{tor}A\to {}_{tor}B$ denote its restriction 
to the respective torsion parts. 
Whether a sheaf is considered in the Zariski topology or in the classical topology will be clear from the context.

%%%%%%%%%
%%%Bloch-Ogus theory
%%%%%%%%%

\section{Preliminaries}\label{sec:preliminaries}

Let $X$ be a smooth complex projective variety of dimension $d$. 
Let $A$ be an abelian group, which will be one of ${\Z}$, ${\Q}$, ${\QZ}$ or ${\Z}/n$ in the sequel. 
Let ${\HH}^q(A)$ be the Zariski sheaf on $Z$ associated to the presheaf $U\mapsto H^q(U, A)$ 
defined by the singular cohomology of Zariski open sets. 
Bloch-Ogus \cite{B-O} computed the $E_2$ page of Grothendieck's coniveau spectral sequence (\cite{Gro} \S 10) 
and obtained a spectral sequence 
\begin{equation*}
E_{2}^{p,q} = H^p(X, {\HH}^q(A)) \; \;  \Rightarrow \; \; H^{p+q}(X, A) 
\end{equation*}
which converges to the coniveau filtration of $H^{k}(X, A)$ 
and which coincides with the Leray spectral sequence for the natural continuous map $X_{cl}\to X_{Zar}$ 
from the classical topology to the Zariski topology. 
They showed that 
\begin{equation*}
H^p(X, {\HH}^q(A)) = 0 \quad \textrm{for} \quad p>q, 
\end{equation*}  
\begin{equation*}
H^p(X, {\HH}^p({\Z})) \simeq {\rm NS}^p(X), 
\end{equation*}  
where ${\rm NS}^p(X)={\rm CH}^p(X)/A^p(X)$ is the group of codimension $p$ cycles modulo algebraic equivalence. 
The edge morphism ${\rm NS}^p(X) \to H^{2p}(X, {\Z})$ equals to the cycle map. 
Note that ${\HH}^q(A)=0$ for $q>d$ because 
smooth affine varieties of dimension $d$ have homotopy type of CW complex of real dimension $d$. 
Hence we also have $E_2^{p,q}=0$ for $q>d$. 
In this way the Bloch-Ogus spectral sequence is confined to the triangle 
\begin{equation*}
p\geq0, \quad q\leq d, \quad p\leq q. 
\end{equation*}

The group $E_2^{0,k}=H^0(X, {\HH}^k(A))$ is usually called the $k$-th unramified cohomology group with coefficients in $A$, 
and denoted by $H^k_{nr}(X, A)$. 
As shown in \cite{CT1}, \cite{BV2}, \cite{CT-V}, for smooth projective $X$, it is birationally invariant, 
and also stably invariant, namely $H^k_{nr}(X, A)\simeq H^k_{nr}(X\times {\proj}^r, A)$. 
When $A={\Z}/n$, $H^k_{nr}(X, {\Z}/n)$ can be computed in terms of the Galois cohomology of the function field of $X$ (see \cite{CT1}), 
which makes $H^k_{nr}(X, {\Z}/n)$ an effective obstruction in the Noether problem.

In this paper we are mainly interested in the cohomology group which sits in the position of ``homology counterpart'' of $H^k_{nr}(X, A)$ 
in the Bloch-Ogus spectral sequence, that is, 
\begin{equation*}
H^{d-k}(X, {\HH}^d(A)). 
\end{equation*}
It is proven by Colliot-Th\'el\`ene and Voisin \cite{CT-V} that this group is also birationally invariant for smooth projective $X$. 
(This can also be seen using the blow-up formula in \cite{BV2}). 
Furthermore, we have 

\begin{lemma}
The group $H^{d-k}(X, {\HH}^d(A))$, $d={\dim}X$, is stably invariant. 
\end{lemma}

\begin{proof}
In \cite{BV2} Barbieri-Viale proved the general formula 
\begin{equation*}
H^p(X\times{\proj}^r, {\HH}^q(A)) \simeq \bigoplus_{i=0}^{r} H^{p-i}(X, {\HH}^{q-i}(A)). 
\end{equation*}
Putting $q=d+r$ and $p=d+r-k$ where $d={\dim}X$, we obtain 
\begin{eqnarray*}
H^{d+r-k}(X\times{\proj}^{r}, {\HH}^{d+r}(A)) 
& \simeq & 
\bigoplus_{i=0}^{r} H^{d+r-k-i}(X, {\HH}^{d+r-i}(A)) \\ 
& = & 
H^{d-k}(X, {\HH}^{d}(A)) 
\end{eqnarray*}
because over $X$ the sheaf ${\HH}^{q'}$ is zero for $q'>d$. 
\end{proof}

Although $H^{d-k}(X, {\HH}^d(A))$ shares some properties with $H^k_{nr}(X, A)$, 
these two groups may be rather different, and $H^{d-k}(X, {\HH}^d(A))$ has some defects compared to $H^k_{nr}(X, A)$: 
for example, $H^{d-k}(X, {\HH}^d({\Z}))$ may have torsion; 
cup product is not defined for $H^{d-k}(X, {\HH}^d(A))$; and 
currently we do not know how to calculate $H^{d-k}(X, {\HH}^d({\Z}/n))$ in terms of the function field of $X$. 
It is well-known that $H^{d-1}(X, {\HH}^d({\QZ}))$ and $H^{d-2}(X, {\HH}^d({\QZ}))$ vanish for rationally connected $X$ 
(cf.~Proposition \ref{H(d-1,d)and H(d-2,d)}). 
According to a conjecture of Voisin, the group $H^{d-3}(X, {\HH}^d({\QZ}))$ 
which by \cite{CT-V} measures the defect of the Hodge conjecture for degree $2$ integral homology classes, 
should also vanish for such $X$ (see \cite{Vo3} and references therein). 
Actually she proves in loc.~cit.~ that this vanishing is implied by the Tate conjecture for degree $2$ Tate classes on surfaces over finite fields.  
By contrast, there are known many examples of rationally connected $X$ for which 
$H^{2}_{nr}(X, {\QZ})={}_{tor}H^3(X, {\Z})$ or $H^{3}_{nr}(X, {\QZ})=Z^4(X)$ is nontrivial (cf.~\cite{CT1}, \cite{CT-V}). 

%Using Bloch-Srinivas' diagonal decomposition principle \cite{B-S}, 
%CT and Voisin \cite{CT-V} proved the following. 

%\begin{proposition}[\cite{CT-V}]\label{prop: H-cohomology torsion}
%Suppose that ${\CH}_0(X)$ is supported in dimension $\leq r$. 
%Then for $k\geq r$, the groups $H^k_{nr}(X, A)$ and $H^{d-k}(X, {\HH}^d(A))$ are annihilated by an integer $N$, 
%in particular torsion. 
%\end{proposition}

Bloch-Srinivas \cite{B-S} proved that the Zariski sheaf ${\HH}^3({\Z})$ is torsion-free using the Merkurjev-Suslin Theorem \cite{M-S}. 
This was generalized to ${\HH}^q({\Z})$ for arbitrary $q$ by Barbieri-Viale \cite{BV3} and Colliot-Th\'el\`ene--Voisin \cite{CT-V}, 
as a consequence of the Bloch-Kato conjecture proved by Rost-Voevodsky \cite{Voe}. 
As a result, we have a short exact sequence of Zariski sheaves (\cite{CT-V}, \cite{BV1}) 
\begin{equation}\label{eqn: Hsheaf torsion-free}
0 \to {\HH}^q({\Z}) \stackrel{n}{\to} {\HH}^q({\Z}) \to {\HH}^q({\Z}/n) \to 0 
\end{equation}
for every $n>0$. 
Taking cohomology long exact sequence, we obtain 

\begin{proposition}[\cite{CT-V}, \cite{BV1}]
For every $p, q\geq0$ and $n>0$ we have an exact sequence 
\begin{equation}\label{eqn:UCT}
0 \to H^p(X, {\HH}^q({\Z}))/n \to H^p(X, {\HH}^q({\Z}/n)) \to {}_{n}H^{p+1}(X, {\HH}^q({\Z})) \to 0. 
\end{equation}
\end{proposition}

For example, when $(p, q)=(0, k)$, this gives an exact sequence 
\begin{equation*}\label{eqn:UCT cohomology edge}
0 \to H^k_{nr}(X, {\Z})/n \to H^k_{nr}(X, {\Z}/n) \to {}_{n}H^{1}(X, {\HH}^k({\Z})) \to 0. 
\end{equation*}
On the mirror edge, namely for $(p, q)=(d-k, d)$, we have an exact sequence 
\begin{equation}\label{eqn:UCT homology edge}
0 \to H^{d-k}(X, {\HH}^d({\Z}))/n \to H^{d-k}(X, {\HH}^d({\Z}/n)) \to {}_{n}H^{d-(k-1)}(X, {\HH}^d({\Z})) \to 0. 
\end{equation}

We remark that the map $H^p(X, {\HH}^q({\Z}/n)) \to {}_{n}H^{p+1}(X, {\HH}^q({\Z}))$ in \eqref{eqn:UCT} 
is induced from the connecting map in the snake lemma applied to the diagram 
%\footnotesize
\begin{equation}\label{eqn:Bockstein CD}
\begin{CD}
@. 0 @.   0  \\
@.   @VVV   @VVV \\ 
\cdots @>>> \bigoplus_{X^{(p)}}H^{q-p}({\C}(x), {\Z}) @>>>   \bigoplus_{X^{(p+1)}}H^{q-p-1}({\C}(x), {\Z}) @>>> \cdots \\ 
@.     @VVnV   @VVnV \\ 
\cdots @>>> \bigoplus_{X^{(p)}}H^{q-p}({\C}(x), {\Z}) @>>>   \bigoplus_{X^{(p+1)}}H^{q-p-1}({\C}(x), {\Z}) @>>> \cdots \\ 
@.    @VVV   @VVV \\ 
\cdots @>>> \bigoplus_{X^{(p)}}H^{q-p}({\C}(x), {\Z}/n) @>>>   \bigoplus_{X^{(p+1)}}H^{q-p-1}({\C}(x), {\Z}/n) @>>> \cdots \\ 
@.     @VVV   @VVV \\ 
@. 0 @.   0. @.  
\end{CD}
\end{equation}
%\normalsize
Here the horizontals are part of the Bloch-Ogus complexes computing the cohomology of ${\HH}^q({\Z})$ and ${\HH}^q({\Z}/n)$ 
($=$ the $q$-th row in the $E_1$ page of the coniveau spectral sequence). 
The columns are exact by the exactness of \eqref{eqn: Hsheaf torsion-free} for smooth Zariski open sets of $\overline{ \{ x \} }$.

Colliot-Th\'el\`ene and Voisin \cite{CT-V}, \cite{Vo2} studied the sequence \eqref{eqn:UCT} in the following cases 
and found connection with algebraic cycles. 
\begin{itemize}
\item $(p, q)=(0, 3)$, 
with the defect $Z^4(X)$ of the integral Hodge conjecture for codimension $2$ cycles \cite{CT-V}; 
\item $(p, q)=(0, 4)$, 
with the torsion Abel-Jacobi kernel $\mathcal{T}^3(X)$ in ${\rm Griff}^3(X)$ \cite{Vo2}; 
\item $(p, q)=(d-3, d)$, 
with the defect $Z_2(X)$ of the integral Hodge conjecture for $1$-cycles \cite{CT-V}. 
\end{itemize} 
In \S \ref{sec:T1(X)} we study the case $(p, q)=(d-4, d)$ and relate it also to $1$-cycles.

%%%%%%%%%
%%%Griff1(X)
%%%%%%%%%

\section{Griffiths group of $1$-cycles}\label{sec:Griff1(X)}

Let $X$ be a smooth projective complex variety of dimension $d$. 
We consider the Bloch-Ogus spectral sequence in degree $\geq2d-4$ and relate it to the Griffiths group of $1$-cycles. 
The situation of the spectral sequence is somewhat similar to that in degree $\leq4$.

\begin{proposition}\label{prop:B-O deg -4 to -2}
We have an exact sequence 
\begin{equation}\label{eqn:B-O deg -4 to -2}
\begin{split}
     & \qquad  \qquad H^{2d-4}(X, {\Z}) \stackrel{e_4}{\to} H^{d-4}(X, {\HdZ}) \\ 
\stackrel{d_2}{\to} \; & H^{d-2}(X, {\HH}^{d-1}({\Z})) \to H^{2d-3}(X, {\Z}) \stackrel{e_3}{\to} H^{d-3}(X, {\HdZ}) \\ 
\stackrel{d_2}{\to} \; & {\NS} \stackrel{cl}{\to} H^{2d-2}(X, {\Z}) \stackrel{e_2}{\to} H^{d-2}(X, {\HdZ}) \to 0. 
\end{split}
\end{equation}
Here ${\NS}={\CH}_1(X)/A_1(X)$ is the group of $1$-cycles modulo algebraic equivalence and 
$cl$ is the cycle map on ${\NS}$.   
The kernel of $e_i$ is the coniveau subgroup $N^{d-i+1}H^{2d-i}(X, {\Z})$ of $H^{2d-i}(X, {\Z})$.   
\end{proposition}

The same exact sequence holds for ${\Z}/n$-coefficients with ${\NS}$ replaced by ${\NS}/n$. 

\begin{proof}
As illustrated in Figure \ref{figure: spectral sequence}, 
\begin{enumerate}
\item the only nonzero differential from degree $2d-3$ to $2d-2$ is $d_2:E_2^{d-3,d}\to E_2^{d-1,d-1}$, 
and there is no nonzero differential from degree $2d-2$ to $2d-1$, 
so the spectral sequence in degree $2d-2$ degenerates at the $E_3$ page;  
\item similarly, the only nonzero differential from degree $2d-4$ to $2d-3$ is $d_2:E_2^{d-4,d}\to E_2^{d-2,d-1}$, 
so the spectral sequence in degree $2d-3$ degenerates at the $E_3$ page as well. 
%Also no nonzero differential arrives at $E_{r}^{d-4,d}$.   
\end{enumerate}
The observation (1) gives an exact sequence 
\begin{equation*}
E_2^{d-3,d} \stackrel{d_2}{\to} E_2^{d-1,d-1} \stackrel{cl}{\to} E_{\infty}^{2d-2} \stackrel{e_2}{\to} E_2^{d-2,d} \to 0,  
\end{equation*}
and (2) gives an exact sequence 
\begin{equation*}
E_{\infty}^{2d-4} \stackrel{e_4}{\to} E_2^{d-4,d} \stackrel{d_2}{\to} E_{2}^{d-2,d-1} \to  
E_{\infty}^{2d-3} \stackrel{e_3}{\to}  E_2^{d-3,d} \stackrel{d_2}{\to} E_2^{d-1,d-1}.  
\end{equation*}
Combining these two exact sequences gives the desired sequence. 
Since the Bloch-Ogus spectral sequence converges to the coniveau filtration, 
the first nontrivial step of the coniveau filtration is ${\ker}(e_i)$ and it is equal to $N^{d-i+1}H^{2d-i}(X, {\Z})$. 
\end{proof}

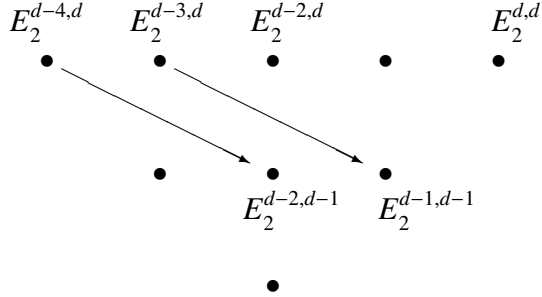
\begin{figure}
\setlength\unitlength{1mm}
\begin{picture}(100,50)(-20,0)
%  \put(15,0){$\bullet$}
  \put(30,0){$\bullet$}
%  \put(0,15){$\bullet$}
  \put(15,15){$\bullet$}
  \put(30,15){$\bullet$}
  \put(45,15){$\bullet$}
  \put(0,30){$\bullet$}
  \put(15,30){$\bullet$}
  \put(30,30){$\bullet$}
  \put(45,30){$\bullet$}
  \put(60,30){$\bullet$}
  
  \put(3,30){\vector(2,-1){25}}
  \put(18,30){\vector(2,-1){25}}
  
  \put(60,35){$E_2^{d, d}$}
  \put(45,10){$E_2^{d-1,d-1}$}
  \put(27,10){$E_2^{d-2,d-1}$}
  \put(-4,35){$E_2^{d-4,d}$}
  \put(12,35){$E_2^{d-3,d}$}
  \put(28,35){$E_2^{d-2,d}$}
\end{picture}
\caption{Bloch-Ogus spectral sequence in degree $\geq 2d-4$}
\label{figure: spectral sequence}
\end{figure}

Colliot-Th\'el\`ene and Voisin \cite{CT-V} studied the third line of \eqref{eqn:B-O deg -4 to -2} and found 
\begin{equation}\label{eqn:H-cohomology and Z2 I}
{}_{n}Z_2(X) \simeq {}_{n}H^{d-2}(X, {\HdZ}). 
\end{equation}
Combining this with \eqref{eqn:UCT homology edge} with $k=3$, they obtained the exact sequence 
\begin{equation}\label{eqn:H-cohomology and Z2 II}
 0 \to H^{d-3}(X, {\HdZ})/n \to H^{d-3}(X, {\Hdn}) \to {}_{n}Z_2(X) \to 0. 
 \end{equation}
  
In \S \ref{sec:T1(X)} we will give a cycle-theoretic interpretation of $H^{d-4}(X, {\Hdn})$. % (Corollary \ref{cor:interpretation of H(d-4,d)}). 
For completeness we also describe $H^{d-k}(X, {\Hdn})$ for $k\leq2$. 

\begin{proposition}\label{H(d-1,d)and H(d-2,d)}
(1) We have exact sequences 
\begin{equation}\label{eqn:H(d-2,d) mod n}
0 \to H^{d-2}(X, {\HH}^d({\Z}))/n \to H^{d-2}(X, {\HH}^d({\Z}/n)) \to {}_{n}H_{1}(X, {\Z}) \to 0, 
\end{equation}
\begin{equation*}\label{eqn:H(d-2,d) tor}
0 \to H^{d-2}(X, {\HH}^d({\Z}))\otimes{\QZ} \to H^{d-2}(X, {\HH}^d({\QZ})) \to {}_{tor}H_{1}(X, {\Z}) \to 0. 
\end{equation*}
When $X$ is rationally connected, %all terms in \eqref{eqn:H(d-2,d) tor} vanish. 
we have $H^{d-2}(X, {\HH}^d({\QZ}))=0$.  

(2) We have 
%\begin{equation*}\label{eqn:H(d-1,d) mod n}
%H^{d-1}(X, {\HH}^d({\Z}/n)) \simeq H^{d-1}(X, {\HH}^d({\Z}))/n \simeq H_{1}(X, {\Z})/n. 
%\end{equation*}
$H^{d-1}(X, {\HH}^d({\Z}/n)) \simeq H_{1}(X, {\Z})/n$. 
This vanishes for rationally connected $X$. 
\end{proposition}

\begin{proof}
These are consequences of \eqref{eqn:UCT homology edge} with $k=1, 2$ 
and degeneration of the Bloch-Ogus spectral sequence in degree $\geq 2d-1$, which gives 
\begin{equation*}
H^{d-1}(X, {\HH}^d({\Z})) \simeq H^{2d-1}(X, {\Z}) \simeq H_{1}(X, {\Z}), 
\end{equation*}
\begin{equation*}
H^{d}(X, {\HH}^d({\Z})) \simeq H^{2d}(X, {\Z}) \simeq {\Z}.
\end{equation*}
When $X$ is rationally connected, it is simply connected and so $H_{1}(X, {\Z})=0$. 
Moreover, $H^{d-2}(X, {\HH}^d({\Z}))\otimes{\QZ}$ is zero by \cite{CT-V} Proposition 3.3. 
\end{proof}

If Voisin's conjecture (cf. \cite{Vo3}) that $Z_2(X)=0$ for rationally connected $X$ holds true, 
we have $H^{d-2}(X, {\HH}^d({\Z}))=0$ for such $X$ by \eqref{eqn:H-cohomology and Z2 I}. 
Then $H^{d-2}(X, {\HH}^d({\Z}/n))$ would also vanish by \eqref{eqn:H(d-2,d) mod n}.

We go back to observing the exact sequence \eqref{eqn:B-O deg -4 to -2}. 
The Griffiths group ${\Griff}$ is the kernel of the cycle map $cl$ defined on ${\NS}$ 
and hence equals to the image of $d_2:H^{d-3}(X, {\HdZ}) \to {\NS}$.  
As in \S \ref{sec:intro}, we denote  
\begin{equation}\label{eqn:Lambda(X)}
\Lambda(X)  :=  H^{2d-3}(X, {\Z})/N^{d-2}H^{2d-3}(X, {\Z}). 
\end{equation}
This is a finitely generated abelian group 
and also isomorphic to  
\begin{equation*}
\begin{split}
& {\rm Im}(\: e_3: H^{2d-3}(X, {\Z}) \to H^{d-3}(X, {\HdZ}) \: )  \\ 
& =   {\ker} (\: d_2: H^{d-3}(X, {\HdZ}) \to {\NS} \: ) 
\end{split} 
\end{equation*}
by \eqref{eqn:B-O deg -4 to -2}. 
Looking at the second to third line of \eqref{eqn:B-O deg -4 to -2}, we obtain the following description of ${\Griff}$.

\begin{proposition}\label{prop:Griff1 and H-cohomology}
We have an exact sequence 
\begin{equation}\label{eqn:Griff1}
0 \to \Lambda(X) \to H^{d-3}(X, {\HdZ}) \to {\Griff} \to 0. 
\end{equation}
In particular, when $H^{2d-3}(X, {\Z})=N^{d-2}H^{2d-3}(X, {\Z})$ holds, we have 
\begin{equation*}
{\Griff} \simeq H^{d-3}(X, {\HdZ}). 
\end{equation*}
\end{proposition}

This is analogous to Bloch-Ogus' description (\cite{B-O}) of ${\rm Griff}^2(X)$ 
\begin{equation*}
0 \to H^3(X, {\Z})/N^1H^3(X, {\Z}) \to H^3_{nr}(X, {\Z}) \to {\rm Griff}^2(X) \to 0. 
\end{equation*}

It is known that the three terms in \eqref{eqn:Griff1} are all birationally invariant for smooth projective $X$. 
For $H^{d-3}(X, {\HdZ})$ this is proved in \cite{CT-V}; 
for the other two terms, this results from the corresponding blow-up formulae: 
\begin{equation*}
{\CH}_1(\tilde{X}) \simeq {\CH}_1(X) \oplus {\CH}_0(Y), 
\end{equation*}
\begin{equation*}
H^{2d-3}(\tilde{X}, {\Z}) \simeq H^{2d-3}(X, {\Z}) \oplus H^{2e-1}(Y, {\Z}),
\end{equation*}
where $\tilde{X}\to X$ is the blow-up along a smooth subvariety $Y\subset X$ of dimension $e$. 

\begin{corollary}
The group ${\Griff}$ is finitely generated if and only if $H^{d-3}(X, {\HdZ})$ is. 
\end{corollary}

Bloch-Srinivas \cite{B-S} proved that 
when ${\CH}_0(X)$ is supported in dimension $\leq2$, ${\Griff}$ is a torsion group. 
Hence in that case, ${\Griff}$ is finite if and only if $H^{d-3}(X, {\HdZ})$ is. 
We remark that the result of Bloch-Srinivas can also be derived from \eqref{eqn:Griff1} and 
the fact that $H^{d-3}(X, {\HdZ})$ is torsion in that case (\cite{CT-V}), but this may be roundabout.

The assumption $H^{2d-3}(X, {\Z})=N^{d-2}H^{2d-3}(X, {\Z})$ in the second statement of 
Proposition \ref{prop:Griff1 and H-cohomology} might be hard to check, unless $H^{2d-3}(X, {\Z})=0$. 
For example, this holds for complete intersections of dimension $d\geq5$. 
When this assumption holds, we have an exact sequence 
\begin{equation*}
0 \to {\Griff}/n \to H^{d-3}(X, {\Hdn}) \to {}_{n}Z_2(X) \to 0 
\end{equation*}
by Colliot-Th\'el\`ene--Voisin's exact sequence \eqref{eqn:H-cohomology and Z2 II}. 
%This gives a cycle-theoretic interpretation of $H^{d-3}(X, {\Hdn})$. 

Taking the ${\rm Ext}({\Z}/n, \cdot )$ long exact sequence associated to \eqref{eqn:Griff1} 
(or equivalently, applying the snake lemma to the multiplication by $n$ on \eqref{eqn:Griff1}), 
we obtain a description of the torsion part of ${\Griff}$. 

\begin{proposition}\label{prop: mod n Griff}
We have exact sequences 
\begin{equation*}
\begin{split}
0 & \to {}_{n}\Lambda(X) \to {}_{n}H^{d-3}(X, {\HdZ}) \to {}_{n}{\Griff} \\ 
 &  \to \Lambda(X)/n \to H^{d-3}(X, {\HdZ})/n \to {\Griff}/n \to 0, 
\end{split}
\end{equation*}
\begin{equation*}\label{eqn:lambdainfty}
\begin{split}
0 & \to {}_{tor}\Lambda(X) \to {}_{tor}H^{d-3}(X, {\HdZ}) \to {}_{tor}{\Griff} \\ 
 &  \to \Lambda(X)\otimes{\QZ} \to H^{d-3}(X, {\HdZ})\otimes{\QZ} \to {\Griff}\otimes{\QZ} \to 0. 
\end{split}
\end{equation*}
\end{proposition}

\begin{corollary}
The group ${}_{n}{\Griff}$ (resp.~${\Griff}/n$) is finite if and only if ${}_{n}H^{d-3}(X, {\HdZ})$ (resp.~$H^{d-3}(X, {\HdZ})/n$) is. 
\end{corollary}

In the corollary we may also replace $H^{d-3}(X, {\HdZ})/n$ by $H^{d-3}(X, {\Hdn})$ 
thanks to the exact sequence \eqref{eqn:H-cohomology and Z2 II}.

%%%%%%%%%
%%%T1(X)
%%%%%%%%%

\section{Torsion Abel-Jacobi kernel in ${\Griff}$}\label{sec:T1(X)}

Let $J^{2d-3}(X)$ be the intermediate Jacobian of $X$ in degree $2d-3$. 
The algebraic part $J^{2d-3}(X)_{alg}$ of $J^{2d-3}(X)$ is defined as 
the image of the Abel-Jacobi map $\lambda_{alg}: A_1(X)\to J^{2d-3}(X)$ from $A_1(X)\subset{\CH}_{1}(X)_{hom}$. 
This is an abelian variety, which corresponds to the weight $1$ sub ${\Q}$-Hodge structure 
$N^{d-2}H^{2d-3}(X, {\Q})$ of $H^{2d-3}(X, {\Q})$. 
Consider the quotient complex torus 
\begin{equation*}
J^{2d-3}(X)_{tr} := J^{2d-3}(X)/J^{2d-3}(X)_{alg}.
\end{equation*}
The Abel-Jacobi map $\lambda: {\CH}_{1}(X)_{hom}\to J^{2d-3}(X)$ from ${\CH}_{1}(X)_{hom}$ induces a homomorphism 
$\lambda_{tr}: {\Griff} \to J^{2d-3}(X)_{tr}$. 
We consider its restriction to the torsion part 
\begin{equation*}
{}_{tor}\lambda_{tr}: {}_{tor}{\Griff} \to {}_{tor}J^{2d-3}(X)_{tr}. 
\end{equation*}
The situation is summarized in the commutative diagram 
\begin{equation*}
\begin{CD}
0 @>>> {}_{tor}A_1(X)                 @>>> {}_{tor}{\CH}_{1}(X)_{hom} @>>>   {}_{tor}{\Griff} @>>> 0 \\ 
@. @V{}_{tor}\lambda_{alg} VV  @VV{}_{tor}\lambda V  @V{}_{tor}\lambda_{tr}VV \\ 
0 @>>> {}_{tor}J^{2d-3}(X)_{alg} @>>> {}_{tor}J^{2d-3}(X)              @>>>   {}_{tor}J^{2d-3}(X)_{tr} @>>> 0. 
\end{CD}
\end{equation*}
%Here ${}_{tor}\lambda_{alg}$ and ${}_{tor}\lambda$ are restriction of $\lambda_{alg}$ and $\lambda$ to the respective torsion parts. 
The two rows are exact by the divisibility of $A_1(X)$ and $J^{2d-3}(X)_{alg}$ respectively. 
The homomorphism ${}_{tor}\lambda_{alg}$ remains surjective 
because $A_{1}(X)$ is generated by abelian varieties via correspondences. 

Following Voisin \cite{Vo2}, we define 
\begin{equation*}
{\T} := {\rm Ker}({}_{tor}\lambda_{tr}) \subset {}_{tor}{\Griff}. 
\end{equation*}
As shown by Voisin, 
${\T}$ coincides with the image of ${\rm Ker}({}_{tor}\lambda)$ in ${\Griff}$,  
which is the definition stated in \S \ref{sec:intro}. 
This equality can be seen from the snake lemma applied to the above commutative diagram, which gives 
the short exact sequence 
\begin{equation}\label{eqn:two def of T1}
0 \to {}_{tor}{\rm Ker}(\lambda_{alg}) \to {}_{tor}{\rm Ker}(\lambda) \to {\T} \to {\rm Coker}({}_{tor}\lambda_{alg}) = 0.   
\end{equation}

Our main result is the following. 

\begin{theorem}\label{thm:main}
We have an exact sequence 
\begin{equation}\label{eqn:main thm 1}
0 \to {}_{tor}\Lambda(X) \to H^{d-4}(X, {\HdQZ})/H^{d-4}(X, {\HdZ})\otimes{\QZ} \to {\T} \to 0. 
\end{equation}
If ${\CH}_0(X)$ is supported in dimension $\leq3$, we have an exact sequence
\begin{equation*}\label{eqn:main thm 2}
0 \to {}_{tor}\Lambda(X) \to H^{d-4}(X, {\HdQZ}) \to {\T} \to 0. 
\end{equation*}
If ${\CH}_0(X)$ is supported in dimension $\leq2$, then $\Lambda(X)$ is finite, 
${\T}$ coincides with ${\Griff}$, and hence we have an exact sequence 
\begin{equation}\label{eqn:main thm 3}
0  \to \Lambda(X)  \to  H^{d-4}(X, {\HdQZ}) \to {\Griff} \to 0. 
\end{equation} 
\end{theorem}

\begin{proof}
When ${\CH}_0(X)$ is supported in dimension $\leq3$, $H^{d-4}(X, {\HdZ})$ is torsion by \cite{CT-V}. 
Hence it is annihilated when tensored with ${\QZ}$. 
When ${\CH}_{0}(X)$ is supported in dimension $\leq2$, ${\Griff}$ is torsion by \cite{B-S},  
and also $H^{2d-3}(X, {\Q})=N^{d-2}H^{2d-3}(X, {\Q})$ 
because $H^{d-3}(X, \mathcal{H}^{d}({\Q}))=0$ by \cite{CT-V}. 
Hence $\Lambda(X)\otimes{\Q}=0$ and $J^{2d-3}(X)_{tr}=0$. 
%so ${\T}$ coincides with ${\Griff}$. 
It remains to prove \eqref{eqn:main thm 1}. 

We apply \eqref{eqn:UCT} and the universal coefficient theorem to the 2nd to 4th terms of \eqref{eqn:B-O deg -4 to -2} 
and its ${\Z}/n$-coefficients version. 
This gives 
%\footnotesize 
\begin{equation}\label{key commutative diagram}
\begin{CD}
0 @>>>  H^{d-4}(X, {\HdZ})/n  @>>>  H^{d-4}(X, {\Hdn})  @>\alpha >>  {}_{n}H^{d-3}(X, {\HdZ})   @>>> 0  \\
@. @VVV  @Vd_2VV  @VV-{}_{n}d_2V \\ 
0 @>>>  H^{d-2}(X, {\HH}^{d-1}({\Z}))/n  @>>>  H^{d-2}(X, {\HH}^{d-1}({\Z}/n))   @>\alpha >>   {}_{n}{\NS} @>>> 0 \\
@. @V\psi/n VV  @V\psi_nVV  @VV-{}_{n}clV \\ 
0 @>>> H^{2d-3}(X, {\Z})/n @>>>    H^{2d-3}(X, {\Z}/n)   @>\beta >>   {}_{n}H^{2d-2}(X, {\Z}) @>>> 0, 
\end{CD}
\end{equation}
%\normalsize
%\footnotesize 
%\begin{equation}\label{key commutative diagram}
%\xymatrix{
%0 \ar[r] & H^{d-4}(X, {\HdZ})/n \ar[r] \ar[d] & H^{d-4}(X, {\Hdn}) \ar[r]^{\alpha} \ar[d]^{d_2} & 
%{}_{n}H^{d-3}(X, {\HdZ}) \ar[r] \ar[d]^{-{}_{n}d_2} & 0 \\ 
%0 \ar[r] & H^{d-2}(X, {\HH}^{d-1}({\Z}))/n \ar[r] \ar[d]^{\psi/n} & 
%H^{d-2}(X, {\HH}^{d-1}({\Z}/n)) \ar[r]^{\alpha} \ar[d]^{\psi_n} & {}_{n}{\NS} \ar[r] \ar[d]^{-{}_{n}cl} & 0 \\ 
%0  \ar[r] &  H^{2d-3}(X, {\Z})/n \ar[r]  &  H^{2d-3}(X, {\Z}/n)  \ar[r]^{\beta}  &  
%{}_{n}H^{2d-2}(X, {\Z})  \ar[r] & 0, \\ 
%}
%\end{equation}
%\normalsize
where all rows and the middle column are exact, and the other two columns are complex. 
The right column is (up to sign) restriction of the $5$th to $7$th terms of \eqref{eqn:B-O deg -4 to -2} to the $n$-torsion parts. 

Commutativity at the lower right, saying that 
the connecting map $\alpha$ of \eqref{eqn:Bockstein CD} with $(p, q)=(d-2, d-1)$ is translated to 
the Bockstein homomorphism $\beta$  
via the edge morphisms of the spectral sequence (up to sign), 
is essentially proved in \cite{CT-S-S} Proposition 1 
(replace $\mu_m$  by ${\Z}$, and the Gersten-Quillen complex by the Bloch-Ogus complex with coefficients in ${\Z}$). 
Commutativity at the upper right, 
namely compatibility of the connecting map $\alpha$ of \eqref{eqn:Bockstein CD} with the $d_2$ differential, 
holds generally. 
It can be checked in the following way (we refer to \cite{CT-H-K} \S 1 for the notation): 
(1) consider the multiplication-by-$n$ and reduction-to-${\Z}/n$-coefficients on the whole exact couple 
$(\underset{\to}{D}^{p,q}, \underset{\to}{E}^{p,q})$ 
with which the coniveau spectral sequence started; 
(2) interpolate  
\begin{equation}\label{eqn:center of commutativity}
{\varinjlim}_{\vec{Z}} H^{p+q+1}(X-Z_{p+3}, X-Z_{p+1})
\end{equation} 
into the relevant diagram  
\begin{equation*}
\xymatrix{
\underset{\to}{E}^{p,q} \ar[r]_k & \underset{\to}{D}^{p+1,q} \ar[r]_j  \ar@{-->}[rd] & \underset{\to}{E}^{p+1,q}  \ar[r]_k & \underset{\to}{D}^{p+2,q}  \\ 
 &                                     & \eqref{eqn:center of commutativity} \ar@{-->}[u] \ar@{-->}[rd]                                &                            \\ 
 & \underset{\to}{D}^{p+2,q-1} \ar[r]_j  \ar[uu]_i  & \underset{\to}{E}^{p+2,q-1} \ar[r]_k \ar@{-->}[u] & \underset{\to}{D}^{p+3,q-1} \ar[r]_j  \ar[uu]_i & \underset{\to}{E}^{p+3,q-1} 
}
\end{equation*}
(multiplication-by-$n$ and reduction-to-${\Z}/n$-coefficients are omitted); 
and then 
(3) run diagram chasing. 
The key point is that both the results of $d_2\circ\alpha$ and $-\alpha\circ d_2$ come from 
a common element of \eqref{eqn:center of commutativity} with ${\Z}$-coefficients which is to be multiplied by $n$. 

Now we apply the snake lemma to the lower two rows of \eqref{key commutative diagram}. 
The resulting connecting map 
${\ker}({}_{n}cl)\to{\coker}(\psi/n)$ is rewritten as  
\begin{equation*}
\lambda_n : {}_{n}{\Griff} \to \Lambda(X)/n. 
\end{equation*}
Then the upper side of \eqref{key commutative diagram} induces the commutative diagram 
%\footnotesize
\begin{equation*}\label{key commutative diagram II}
\begin{CD}
0 @>>>  H^{d-4}(X, {\HdZ})/n  @>>>  H^{d-4}(X, {\Hdn})  @>>>  {}_{n}H^{d-3}(X, {\HdZ})   @>>> 0  \\
@. @VVV  @Vd_2VV  @VV\pi_nV \\ 
0 @>>>  {\ker}(\psi/n)  @>>>  {\ker}(\psi_n)   @>>>   {\ker}(\lambda_n) @>>> 0.  
\end{CD}
\end{equation*}
%\normalsize
%\footnotesize
%\begin{equation*}
%\xymatrix{
%0 \ar[r] & H^{d-4}(X, {\HdZ})/n \ar[d] \ar[r] & H^{d-4}(X, {\Hdn}) \ar[d]_{d_{2}} \ar[r] & 
%{}_{n}H^{d-3}(X, {\HdZ}) \ar[d]_{\pi_n} \ar[r] & 0 \\ 
%0 \ar[r] & \; \; \; \; {\ker}(\psi/n) \; \; \; \; \ar[r]  & \; \; \; \; \; \; {\ker}(\psi_n) \; \; \; \; \; \; \ar[r] 
%& \; \; \; \; {\ker}(\lambda_n) \; \; \; \; \ar[r] & 0. \\
%}
%\end{equation*}
%\normalsize
Since the middle vertical is surjective by the exactness of the middle column of \eqref{key commutative diagram}, 
the right vertical $\pi_n$ is also surjective. 
Since $\pi_n$ is restriction of the map  
$d_2:H^{d-3}(X, {\HdZ}) \to {\NS}$ of \eqref{eqn:B-O deg -4 to -2} 
to the $n$-torsion part, 
its kernel is 
\begin{equation*}
{}_{n}{\ker}(\: d_2:H^{d-3}(X, {\HdZ}) \to {\NS}\: ) \simeq {}_{n}\Lambda(X).
\end{equation*}
Thus we obtain an exact sequence 
\begin{equation*}
0 \to {}_{n}\Lambda(X) \to {}_{n}H^{d-3}(X, {\HdZ}) \to {\ker}(\lambda_n) \to 0. 
\end{equation*}
Taking direct limit with respect to $n$, we have an exact sequence  
\begin{equation*}
0 \to {}_{tor}\Lambda(X) \to {}_{tor}H^{d-3}(X, {\HdZ}) \to {\ker}(\lambda_{\infty}) \to 0 
\end{equation*}
where 
\begin{equation*}
\lambda_{\infty} := \varinjlim_{n} \lambda_n : {}_{tor}{\Griff} \to \Lambda(X)\otimes{\QZ}. 
\end{equation*}
The construction of $\lambda_{\infty}$ as a connecting map coincides with that of Voisin's map $cl_{d-1,tors,tr}$ in \cite{Vo2} p.354 
(more precisely, its restriction to ${}_{tor}{\Griff}$). 
By \cite{Vo2} Proposition 4.4, we see that  $\lambda_{\infty}={}_{tor}\lambda_{tr}$ and so ${\ker}(\lambda_{\infty})={\T}$. 
Finally, our assertion follows by taking direct limit of \eqref{eqn:UCT homology edge} with $k=4$. 
\end{proof}

\begin{remark}\label{remark}
(1) This proof was inspired by the argument of Voisin \cite{Vo2} in the $H^4_{nr}$ case. 
On the other hand, if we add the present argument to Voisin's proof, we obtain an exact sequence 
%\footnotesize
\begin{equation*}
0 \to {}_{tor}(N^1H^5(X, {\Z})/N^2H^5(X, {\Z})) \to H^4_{nr}(X, {\QZ})/H^4_{nr}(X, {\Z})\otimes{\QZ} \to \mathcal{T}^3(X) \to 0. 
\end{equation*}
%\normalsize
Since $H^5(X, {\Z})/N^1H^5(X, {\Z})$ has no torsion (\cite{CT-V}, \cite{BV3}), the first term is equal to 
${}_{tor}(H^5(X, {\Z})/N^2H^5(X, {\Z}))$.     
This gives a refinement of Voisin's result which was proved under the assumption that $H^5(X, {\Z})/N^2H^5(X, {\Z})$ has no torsion. 

(2) We can rewrite \eqref{eqn:main thm 1} in a form compatible with \eqref{eqn:Griff1}: 
\begin{equation*}
\xymatrix{
0 \ar[r] & {}_{tor}\Lambda(X) \ar[r] \ar@{^{(}->}[d]  & {}_{tor}H^{d-3}(X, {\HdZ}) \ar[r] \ar@{^{(}->}[d]   & {\T} \ar[r] \ar@{^{(}->}[d]  & 0 \\ 
0 \ar[r]   &     \Lambda(X) \ar[r]         &            H^{d-3}(X, {\HdZ}) \ar[r]          &           {\Griff} \ar[r]         &       0.
}
\end{equation*}
It seems plausible that the map $\lambda_n$ coincides with the connecting map in 
Proposition \ref{prop: mod n Griff}, but we have not checked this. 
%In this direction \eqref{eqn:main thm 1} would be derived more directly. 

(3) The exact sequence \eqref{eqn:main thm 3} also follows from \eqref{eqn:Griff1} 
and the vanishing $H^{d-i}(X, {\HH}^d({\Q}))=0$ for $i=3, 4$. 
\end{remark}

In the proof of Theorem \ref{thm:main} we also obtained a description of 
the kernel of the ``mod $n$'' transcendental Abel-Jacobi map 
$\lambda_n: {}_{n}{\Griff} \to \Lambda(X)/n$ in terms of  $H^{d-4}(X, {\Hdn})$.

\begin{corollary}\label{cor:interpretation of H(d-4,d)}
We have an exact sequence 
\begin{equation*}
0 \to {}_{n}\Lambda(X) \to H^{d-4}(X, {\Hdn})/(H^{d-4}(X, {\HdZ})/n) \to {\ker}(\lambda_n) \to 0. 
\end{equation*}
\end{corollary}

The referee pointed out that 
in some case, algebraic equivalence for torsion $1$-cycles with null Deligne class 
reduces to rational equivalence. 
Consider the condition 
\begin{equation}\label{eqn:Z-coniveau surjectivity}
\bigoplus_{W\in X^{(d-2)}} H^{1}(\tilde{W}, {\Z}) \to 
N^{d-2}H^{2d-3}(X, {\Z})/\textrm{torsion} 
\quad \textrm{is surjective}, 
\end{equation}
where $W$ runs through irreducible surfaces in $X$ 
and $\tilde{W}$ is a desingularization of $W$. 
As noticed by Grothendieck \cite{Gro2}, 
this seems a subtle problem while its ${\Q}$-coefficients version always holds.

\begin{proposition}\label{prop:alg=rat equiv}
Suppose that 
(i) ${\CH}_{0}(X)$ is supported in dimension $0$, 
(ii) $\Lambda(X)$ is torsion-free, and 
(iii) the condition \eqref{eqn:Z-coniveau surjectivity} holds. 
Then the map 
${}_{tor}{\rm Ker}(\lambda) \to {\T}$ 
in \eqref{eqn:two def of T1} is isomorphic. 
%In other words, for torsion $1$-cycles with null Deligne class, 
%algebraic and rational equivalence coincide. 
Thus in this case we have 
\begin{equation*}
H^{d-4}(X, {\HdQZ}) \simeq {\Griff} = {\T} \simeq {}_{tor}{\rm Ker}(\lambda), 
\end{equation*}
the last being a subgroup of ${}_{tor}{\CH}_1(X)_{hom}$. 
\end{proposition}

\begin{proof}
The second assertion follows from the first and Theorem \ref{thm:main}. 
By \eqref{eqn:two def of T1}, the first assertion is equivalent to 
the vanishing of ${}_{tor}{\rm Ker}(\lambda_{alg})$. 
We first show that 
the conditions (ii) and (iii) imply that ${\rm Ker}(\lambda_{alg})$ is a divisible group. 
Since $A_{1}(X)$ is divisible, it suffices to prove that 
${}_{n}A_{1}(X)\to {}_{n}J^{2d-3}(X)_{alg}$ is surjective for every $n>0$. 
As in the proof of \cite{Blo1} Lemma 1.4, 
this follows if we could find a (possibly reducible) surface $W\subset X$ such that 
${}_{n}{\rm Pic}^{0}(\tilde{W}) \to {}_{n}J^{2d-3}(X)_{alg}$ is surjective 
where $\tilde{W}$ is a desingularization of $W$, 
for this homomorphism is factorized as 
\begin{equation*}
{}_{n}{\rm Pic}^{0}(\tilde{W}) = {}_{n}A_{1}(\tilde{W}) 
\to {}_{n}A_{1}(X) \to {}_{n}J^{2d-3}(X)_{alg}. 
\end{equation*}
Since ${\rm Pic}^{0}(\tilde{W}) \to J^{2d-3}(X)_{alg}$ is defined by  
the morphism of ${\Z}$-Hodge structure of weight $1$ 
\begin{equation*}
j_{\ast} : H^{1}(\tilde{W}, {\Z}) \to N^{d-2}H^{2d-3}(X, {\Q})\cap H^{2d-3}(X, {\Z}) 
\end{equation*}
where $j:\tilde{W}\to X$, 
our assertion holds if $j_{\ast}$ is surjectve. 
Note that $j_{\ast}$ factors through 
\begin{equation*}
H^{1}(\tilde{W}, {\Z}) \to 
N^{d-2}H^{2d-3}(X, {\Z})/\textrm{torsion}  \subset 
N^{d-2}H^{2d-3}(X, {\Q})\cap H^{2d-3}(X, {\Z}). 
\end{equation*}
By the condition (ii), the second inclusion is equality. 
By the condition (iii), we can find finitely many irreducible surfaces $W_{1}, \cdots, W_{k}\subset X$ 
such that $\oplus_{i}H^{1}(\tilde{W}_{i}, {\Z})$ maps surjectively onto $N^{d-2}H^{2d-3}(X, {\Z})/\textrm{torsion}$. 
It now suffices to take $W=\sum_{i}W_{i}$. 

We next show that the condition (i) implies that 
${}_{tor}{\rm Ker}(\lambda_{alg})$ is annihilated by some natural number. 
We prove this for ${}_{tor}{\rm Ker}(\lambda)$. 
By the Bloch-Srinivas principle \cite{B-S}, we have a decomposition of the diagonal 
\begin{equation*}
N\Delta_{X} \sim \Gamma_1 + \Gamma_2 \quad  \in {\CH}^d(X\times X) 
\end{equation*}
for some $N>0$, 
where $\Gamma_1$ is supported on $D\times X$ for some divisor $D\subset X$ 
and $\Gamma_2$ is supported on $X\times \{ p_1, \cdots, p_k \}$ 
for some points $p_1, \cdots, p_k \in X$. 
The correspondence by $\Gamma_2$ clearly annihilates every $1$-cycle. 
On the other hand, if $\tilde{D}\to D$ is a desingularization of $D$, 
the correspondence by $\Gamma_1$ factors through the pullback to $\tilde{D}$. 
The pullback of a torsion $1$-cycle with null Deligne class to $\tilde{D}$ 
is a torsion $0$-cycle with null Albanese invariant, 
which must be rationally equivalent to zero by a theorem of Roitman \cite{Ro}. 
Hence $N\cdot{}_{tor}{\rm Ker}(\lambda)=0$. 
We thus obtain ${}_{tor}{\rm Ker}(\lambda_{alg})=0$. 
\end{proof}

For example, this proposition applies to Fano complete intersections of dimension $d\geq5$ 
for they are rationally connected and have $H^{2d-3}(X, {\Z})=0$. 
We note that to deduce divisibility of ${\rm Ker}(\lambda_{alg})$, 
one may also replace the conditions (ii), (iii) by 
(iv) $N^{d-2}H^{2d-3}(X, {\Q})=0$ because ${\rm Ker}(\lambda_{alg})=A_{1}(X)$ in that case.

%%%%%%%%%
%%%Torsion in Ak(X)
%%%%%%%%%

\section{Torsion cycles in $A^k(X)$}\label{sec:nAk(X)}

The previous sections dealt with ${\Griff}$ and its torsion subgroup. 
In this section, as another application of \eqref{eqn:UCT}, 
we give a description of the torsion part of $A^p(X)$ in terms of ${\HH}$-cohomology for arbitrary $p$. 
This is independent of \S \ref{sec:Griff1(X)} and \S \ref{sec:T1(X)}. 
Let $\mathcal{K}_{p}$ and $\mathcal{K}_{p}^{M}$ be the Zariski sheaves on $X$ 
associated to the Quillen $K$-theory and the Milnor $K$-theory respectively. 
As a consequence of the Bloch-Kato conjecture proved by Rost-Voevodsky \cite{Voe} and the work of Kerz \cite{Ke},  
we have an isomorphism $\mathcal{K}_{p}^{M}/n\simeq {\HH}^p({\Z}/n)$ for every $n>1$ 
(see \cite{CT-V} p.745 and \cite{BV3} p.9). 

\begin{theorem}\label{prop: torsion A^p(X) and H-cohomology}
Let $X$ be a smooth complex projective variety. 
We have exact sequences 
\begin{equation*}
0 \to H^{p-1}(X, \mathcal{K}_p)/n \to H^{p-1}(X, {\HH}^p({\Z}))/n \to {}_{n}A^p(X) \to 0, 
\end{equation*}
\begin{equation*}\label{eqn:n_A and H-cohomology II}
0 \to H^{p-1}(X, \mathcal{K}_p)\otimes{\QZ} \to H^{p-1}(X, {\HH}^p({\Z}))\otimes{\QZ} \to {}_{tor}A^p(X) \to 0. 
\end{equation*}
The same exact sequences with $\mathcal{K}_{p}$ replaced by $\mathcal{K}_{p}^{M}$ also hold. 
\end{theorem}
 
\begin{remark}
%(1) As explained in the proof, we may also replace $\mathcal{K}_p$ by the Milnor $K$-theory sheaf $\mathcal{K}_{p}^{M}$. 
%Indeed, we have $H^{p-1}(X, \mathcal{K}_p)\simeq H^{p-1}(X, \mathcal{K}_p^{M})$. 
%
%(2) 
When $p=2$, the second sequence reduces to the famous formula (\cite{M-S}, \cite{Mu}) 
\begin{equation*}
{}_{tor}A^2(X)  \simeq H^{1}(X, {\HH}^2({\Z}))\otimes{\QZ}  
 \simeq N^1H^3(X, {\Z}) \otimes{\QZ} \simeq {}_{tor}J^3(X)_{alg}. 
\end{equation*}
\end{remark}

\begin{proof}
We will combine three short exact sequences. 
Firstly, the $(p-1, p)$ case of \eqref{eqn:UCT} gives an exact sequence 
\begin{equation}\label{eqn:ses I}
0 \to H^{p-1}(X, {\HH}^p({\Z}))/n \to H^{p-1}(X, {\HH}^p({\Z}/n)) \to {}_{n}{\rm NS}^p(X) \to 0. 
\end{equation}
As explained at \eqref{eqn:Bockstein CD}, 
the map $H^{p-1}(X, {\HH}^p({\Z}/n)) \to {}_{n}{\rm NS}^p(X)$ is induced from the boundary map in the snake lemma for 
the middle and left columns of the diagram 
\begin{equation}\label{eqn:CD I}
\begin{CD}
\bigoplus_{X^{(p-2)}}H^2({\C}(x), {\Z}) @>>>   \bigoplus_{X^{(p-1)}}H^1({\C}(x), {\Z}) @>>> \bigoplus_{X^{(p)}} {\Z} \\ 
@VVnV     @VVnV   @VVnV \\ 
\bigoplus_{X^{(p-2)}}H^2({\C}(x), {\Z}) @>>>   \bigoplus_{X^{(p-1)}}H^1({\C}(x), {\Z}) @>>> \bigoplus_{X^{(p)}} {\Z} \\ 
@VVV     @VVV   @VVV \\ 
\bigoplus_{X^{(p-2)}}H^2({\C}(x), {\Z}/n) @>>>   \bigoplus_{X^{(p-1)}}H^1({\C}(x), {\Z}/n) @>>> \bigoplus_{X^{(p)}} {\Z}/n \\
@VVV     @VVV   @VVV \\ 
0 @.   0 @. 0. 
\end{CD}
\end{equation}

Secondly, Colliot-Th\'el\`ene, Sansuc and Soul\'e derived an exact sequence (\cite{CT-S-S} and \cite{CT0} \S 3.2)
\begin{equation}\label{eqn:ses II}
0 \to H^{p-1}(X, \mathcal{K}_p)/n \to H^{p-1}(X, {\HH}^p({\Z}/n)) \to {}_{n}{\CH}^p(X) \to 0. 
\end{equation}
As explained in \cite{CT-S-S}, \cite{CT0}, 
the map $H^{p-1}(X, {\HH}^p({\Z}/n)) \to {}_{n}{\CH}^p(X)$ is induced from the boundary map in the snake lemma for 
the middle and left columns of the diagram 
\begin{equation}\label{eqn:CD II}
\begin{CD}
\bigoplus_{X^{(p-2)}}K_2{\C}(x) @>>>   \bigoplus_{X^{(p-1)}}{\C}(x)^{\times} @>>> \bigoplus_{X^{(p)}} {\Z} \\ 
@VVnV     @VVnV   @VVnV \\ 
\bigoplus_{X^{(p-2)}}K_2{\C}(x) @>>>   \bigoplus_{X^{(p-1)}}{\C}(x)^{\times} @>>> \bigoplus_{X^{(p)}} {\Z} \\ 
@VVV     @VVV   @VVV \\ 
\bigoplus_{X^{(p-2)}}H^2({\C}(x), {\Z}/n) @>>>   \bigoplus_{X^{(p-1)}}H^1({\C}(x), {\Z}/n) @>>> \bigoplus_{X^{(p)}} {\Z}/n. \\
@VVV     @VVV   @VVV \\ 
0 @.   0 @. 0. 
\end{CD}
\end{equation}
The upper two rows come from the Gersten-Quillen resolution of $\mathcal{K}_p$. 
%(we may replace $\mathcal{K}_p$ by $\mathcal{K}_p^M$ 
%because the Gersten resolution has been established also for the Milnor $K$-theory \cite{Ke} 
%and $K_iF=K_i^MF$ for $i\leq2$). 
The columns are exact: 
the middle comes from the Kummer theory, 
and the left from the Merkurjev-Suslin theorem \cite{M-S}. 

Note that in \eqref{eqn:ses II} we may replace $\mathcal{K}_p$ by $\mathcal{K}_p^M$   
because the Gersten resolution has been established also for the Milnor $K$-theory \cite{Ke} and $K_iF=K_i^MF$ for $i\leq2$,  
so the upper rows also compute $H^{p-1}(X, \mathcal{K}_p^{M})$. 
Indeed, we have $H^{p-1}(X, \mathcal{K}_p)\simeq H^{p-1}(X, \mathcal{K}_p^{M})$. 

Thirdly, since $A^p(X)$ is divisible, the sequence 
\begin{equation}\label{eqn:ses III}
0 \to {}_{n}A^p(X) \to {}_{n}{\CH}^p(X) \to {}_{n}{\rm NS}^p(X) \to 0 
\end{equation}
remains exact. 

Combining the three exact sequences \eqref{eqn:ses I}, \eqref{eqn:ses II}, \eqref{eqn:ses III}, we obtain 
%\footnotesize
\begin{equation*}\label{eqn:cdIII}
\begin{CD}
     @. @.  0 @. \\ 
   @.  @.  @VVV \\ 
0 @>>> H^{p-1}(X, \mathcal{K}_p)/n @= H^{p-1}(X, \mathcal{K}_p)/n @>>> 0 @.    \\
 @.   @.  @VVV  @VVV \\ 
0 @>>> H^{p-1}(X, {\HH}^p({\Z}))/n @>>> H^{p-1}(X, {\HH}^p({\Z}/n)) @>>> {}_{n}{\rm NS}^p(X) @>>>  0    \\
  @.  @.  @VVV  @| \\ 
0 @>>>  {}_{n}A^p(X) @>>> {}_{n}{\CH}^p(X) @>>> {}_{n}{\rm NS}^p(X) @>>>  0    \\
   @.  @. @VVV  \\ 
@. @.  0 
\end{CD}
\end{equation*}
%\normalsize
%\begin{equation*}
%\xymatrix{
% & & 0 \ar[d] & &  \\ 
%0 \ar[r] & H^{p-1}(X, \mathcal{K}_p)/n \ar[r] & 
%H^{p-1}(X, \mathcal{K}_p)/n \ar[r] \ar[d] & 0 \ar[d] &  \\ 
%0 \ar[r] & H^{p-1}(X, {\HH}^p({\Z}))/n \ar[r] &
% H^{p-1}(X, {\HH}^p({\Z}/n)) \ar[r] \ar[d] & {}_{n}{\rm NS}^p(X) \ar[r] \ar[d] &  0 \\ 
%0 \ar[r] & \; \; \;  {}_{n}A^p(X) \; \; \; \ar[r] & 
%\; \; \; {}_{n}{\CH}^p(X) \; \; \; \ar[r] \ar[d] & \; {}_{n}{\rm NS}^p(X) \; \ar[r] \ar[d] & 0 \\ 
% & & 0 & 0 & \\ 
%}
%\end{equation*}
Commutativity at the lower right can be checked  
by comparing the boundary maps in the snake lemmas for \eqref{eqn:CD I} and \eqref{eqn:CD II}.
They are connected through the map 
${\C}(x)^{\times}\to H^1({\C}(x), {\Z})$ 
which is induced from the boundary maps of the exponential sequences on 
smooth Zariski open sets of $\overline{ \{ x \} }$.  
Now diagram chasing (or the snake lemma)
induces vertical morphisms on the left column of this diagram, and shows that it is exact. 
\end{proof}

By Proposition \ref{prop: mod n Griff} and 
Theorem \ref{prop: torsion A^p(X) and H-cohomology} with $p=d-1$, 
we conclude that the torsion part of ${\CH}_1(X)_{hom}$ is controlled by 
the degree $2d-3$ terms in the $E_2$ page of 
the Bloch-Ogus spectral sequence with ${\Z}$-coefficients. 

%%%%%%% Reference %%%%%%%%%%%%%%%%%%%%%%%%%%%%%


\begin{thebibliography}{99}

\bibitem{BV1}
Barbieri-Viale, L.
\textit{Cicli di codimensione 2 su variet\`a unirazionali complesse.} 
in "K-Theory (Strasbourg 1992)", Ast\'erisque \textbf{226} (1994), 13--41. 

\bibitem{BV2}
Barbieri-Viale, L.
\textit{$\mathcal{H}$-cohomologies versus algebraic cycles.} 
Math. Nachr. \textbf{184} (1997), 5--57. 

\bibitem{BV3}
Barbieri-Viale, L.
\textit{On the Deligne-Beilinson cohomology sheaves.} 
 Ann. K-Theory \textbf{1}, no.1 (2016), 3--17. 

\bibitem{Blo1}
Bloch, S.
\textit{Lectures on algebraic cycles.} 
(Second edition) Cambridge University Press, 2010. 

\bibitem{B-O}
Bloch, S.; Ogus, A.
\textit{Gersten's conjecture and the homology of schemes.} 
Ann. Sci. \'Ecole Norm. Sup. (4) \textbf{7} (1974), 181--201. 

\bibitem{B-S}
Bloch, S.; Srinivas, V.
\textit{Remarks on correspondences and algebraic cycles.} 
Amer. J. Math. \textbf{105} (1983), no. 5, 1235--1253. 

\bibitem{CT0}
Colliot-Th\'el\`ene, J.-L.
\textit{Cycles alg\'ebriques de torsion et K-th\'eorie alg\'ebrique.} 
in ``Arithmetic algebraic geometry (Trento, 1991)'', 1--49, 
Lecture Notes in Math., \textbf{1553}, Springer, 1993. 

\bibitem{CT1}
Colliot-Th\'el\`ene, J.-L.
\textit{Birational invariants, purity and the Gersten conjecture.} 
in "K-theory and algebraic geometry: connections with quadratic forms and division algebras (Santa Barbara, CA, 1992)", 1--64, 
Proc. Symp. Pure Math., \textbf{58}, Part 1, Amer. Math. Soc., 1995. 

\bibitem{CT-H-K}
Colliot-Th\'el\`ene, J.-L.; Hoobler, R.-T.; Kahn, B.
\textit{The Bloch-Ogus-Gabber theorem.} 
in ``Algebraic K-theory (Toronto, ON, 1996)'', 31--94, 
Fields Inst. Commun., \textbf{16}, Amer. Math. Soc., 1997. 

\bibitem{CT-S-S}
Colliot-Th\'el\`ene, J.-L.; Sansuc, J.-J.; Soul\'e, C.
\textit{Torsion dans le groupe de Chow de codimension deux.} 
Duke Math. J. \textbf{50} (1983), no. 3, 763--801. 

\bibitem{CT-V}
Colliot-Th\'el\`ene, J.-L.; Voisin, C.
\textit{Cohomologie non ramifi\'ee et conjecture de Hodge enti\`ere.}
Duke Math. J. \textbf{161} (2012), no. 5, 735--801. 

\bibitem{Gro}
Grothendieck, A. 
\textit{Le groupe de Brauer. III.}  
in ``Dix expos\'es sur la cohomologie des sch\'emas'', 88--188, 
North-Holland, 1968. 
%\textit{On the de Rham cohomology of algebraic varieties.} 
%Inst. Hautes \'Etudes Sci. Publ. Math. No. 29 (1966) 95--103. 

\bibitem{Gro2}
Grothendieck, A.
\textit{Hodge's general conjecture is false for trivial reasons.} 
Topology \textbf{8} (1969) 299--303. 

\bibitem{Ka} 
Kahn, B.
\textit{Applications of weight-two motivic cohomology.} 
Doc. Math. \textbf{1} (1996), No. 17, 395--416. 

\bibitem{Ke}
Kerz, M.
\textit{The Gersten conjecture for Milnor K-theory.} 
Invent. Math. \textbf{175} (2009), no. 1, 1--33. 

\bibitem{M-S}
Merkurjev, A. S.; Suslin, A. A.
\textit{K-cohomology of Severi-Brauer varieties and the norm residue homomorphism.} 
Math. USSR-Izv. \textbf{21} (1983), no. 2, 307--340. 

\bibitem{Mu}
Murre, J. P. 
\textit{Un r\'esultat en th\'eorie des cycles alg\'ebriques de codimension deux.} 
C. R. Acad. Sci. Paris S\'er. I  \textbf{296} (1983), no. 23, 981--984. 

\bibitem{Pa}Parimala, R.
\textit{Witt groups of affine three-folds.} 
Duke Math. J. \textbf{57} (1988), no. 3, 947--954. 
%\textit{Witt groups vis-\`a-vis Chow groups.} 
%Proceedings of the Indo-French Conference on Geometry (Bombay, 1989), 149--154, Hindustan Book, 1993. 

\bibitem{Ro}Roitman, A. A.
\textit{The torsion of the group of $0$-cycles modulo rational equivalence.} 
Ann. of Math. (2) \textbf{111} (1980), no. 3, 553--569. 

\bibitem{Voe}
Voevodsky, V.
\textit{On motivic cohomology with ${\Z}/l$-coefficients.}
Ann. of Math. (2) \textbf{174} (2011), no. 1, 401--438. 

%\bibitem{Vo1}
%Voisin, C.
%\textit{Some aspects of the Hodge conjecture.} 
%Jpn. J. Math. \textbf{2} (2007), no. 2, 261--296. 

\bibitem{Vo2}
Voisin, C.
\textit{Degree 4 unramified cohomology with finite coefficients and torsion codimension 3 cycles.} 
in "Geometry and arithmetic", 347--368, 
EMS Ser. Congr. Rep., Eur. Math. Soc., 2012. 

\bibitem{Vo3}
Voisin, C.
\textit{Remarks on curve classes on rationally connected varieties.} 
in "A celebration of algebraic geometry", 591--599, 
Clay Math. Proc., \textbf{18}, Amer. Math. Soc., 2013. 

\end{thebibliography}
\end{document}